\documentclass{aic}

\usepackage{stmaryrd}
\usepackage{bm}
\usepackage{newtxtext,newtxmath}

\theoremstyle{plain}
\newtheorem{thm}{Theorem}
\newtheorem{claim}[thm]{Claim}

\newtheorem{lem}[thm]{Lemma}
\newtheorem{prop}[thm]{Proposition}

\aicAUTHORdetails{%
  title = {On Partitioning the Edges of an Infinite Digraph into Directed Cycles}, %% please capitalize all significant words
  author = {Attila Jo\'o},
  plaintextauthor = {Attila Joo},
  keywords = {infinite digraph, directed cycle, edge-partition},
}   %%% END \aicAUTHORdetails

\aicEDITORdetails{%
   year={2021},
   number={2},
   received={15 October 2018},   % received date: example: 7 January 2017
   revised={10 July 2020},    % Optional revised date (you may comment it out)
   published={15 January 2021},  % published date
   doi={10.19086/aic.18702},      % XXX = number of paper, e.g. aic006 for paper#6
}   %%% END \aicEDITORdetails

\begin{document}

\begin{frontmatter}[classification=text]
%% EDITOR: this will force the keywords to appear right after the Abstract.
%%   If the abstract is too long and would force the keywords off the
%%   front page, please comment out % [classification=text] above
%%   This way the keywords will be floated on the bottom of the first page
%%   even though the Abstract spills over to the next page.

%%% AUTHOR: Title goes here.  This line is optional.  You must use it
%%   if title has footnote attached or requires nontrivial typesetting,
%%   e.g., inclusion of linebreaks to force nice layout.
%\title{Short Proof of R\"odl's $n^{\log\log n}$ Bound\titlefootnote{This is a footnote to the title}} %% please capitalize all 
%%%significant words

%%% AUTHOR:
%%% List all authors. If you wish, place grant acknowledgements in \thanks.
%%% In brackets include a short tag for each author.
\author[jojo]{Attila Jo\'o \thanks{Funding was provided by the
Alexander von Humboldt Foundation and by NKFIH OTKA-129211.}}

%%% AUTHOR: Abstract goes here
\begin{abstract}
 Nash-Williams proved in \cite{nash1960decomposition} that for an undirected graph $ G $ the set $ E(G) $ can be partitioned into cycles if and only if there is no finite cut of odd size. Later C. Thomassen gave a simpler proof for this in \cite{thomassen2017nash} and conjectured  the following 
 directed analogue of the theorem: the edge set 
 of a digraph can be partitioned into directed cycles  if and only if for each subset of the vertices 
 the cardinality of the ingoing and the outgoing edges are equal. The aim of the paper is to prove this conjecture.
\end{abstract}
\end{frontmatter}

%%% AUTHOR: body of paper starts here
\section{Introduction}
 One of Nash-Williams'  famous results  in infinite graph theory is the following:

\begin{thm}[Nash-Williams, \cite{nash1960decomposition} (p. 235 Theorem 3)]\label{nwundi}
If $ G $ is an undirected graph, then $ E(G) $ can be partitioned into cycles if and only if there is no finite cut of odd size.
\end{thm}

 After giving a simpler proof for Theorem \ref{nwundi}, C. Thomassen conjectured a 
 directed version of it (see \cite{thomassen2017nash} p. 1037). Since the main result of our paper is deciding this 
 conjecture positively, we state it as a theorem.
 
 \begin{thm}\label{main thm}
   If $ D=(V,E) $ is a directed graph, then $ E $ can be partitioned into directed cycles if and only if 
    for all $ X\subseteq V $ the cardinalities of the set of the ingoing and the outgoing edges of $ X $ are equal.
   \end{thm}

It is worth to mention that already Nash-Williams himself claimed Theorem \ref{main thm}  to be true in
 \cite{nash1960decomposition}.  L. Soukup gave a new shorter proof for Theorem 
 \ref{nwundi}   (see Theorem 5.1 of \cite{soukup2011elementary}) based 
on elementary submodels. The main difficulty of the proof of Theorem \ref{main thm} compared to the undirected variant is the 
following. The obstacle for the cycle partition in Theorem \ref{nwundi} is a finite set (an odd cut)   but the obstacle usually fails to be 
finite in 
the context of  Theorem \ref{main thm}. 

The paper is designed to be comprehensible for everybody with a basic familiarity in 
infinite combinatorics. No advanced set theoretic concepts are used. Our main tool is the elementary submodel method and 
although it is relatively common in this field, we introduce it briefly. A more detailed 
introduction to elementary submodels where not even basic logic background (first order formulas and models) is assumed can be found in \cite{soukup2011elementary} with several applications in infinite combinatorics including a proof of 
Theorem \ref{nwundi}.  

%\newpage %% AUTHOR: please comment out this line.  It serves only
%%   to demonstrate both types of header line in aic-template.pdf

\section{Notation}
The digraphs $ D=(V,E) $ in the paper may have loops and parallel edges.\footnote{One can reduce Theorem \ref{main thm} to 
the case of 
simple digraphs by the subdivision of parallel edges with a new vertex together with the deletion of loops but it would not 
make any difference in our proof.}  For a 
subset $ X $ of $ V $, we denote 
by $ \bm{\mathsf{out}_D(X)}$ and  $\bm{\mathsf{in}_D(X)} $  the set 
of outgoing and ingoing edges of $ X $  in $ D $ respectively\footnote{Edge $ e $ is ingoing  with respect to $ X $ if its head is in $ 
X $ but its tail is not.}
and let 
$ \bm{\mathsf{cut}_D(X)}:=\mathsf{out}_D(X)\cup \mathsf{in}_D(X) $.  For an  $ X\subseteq V $ let $ 
\bm{D[X]} $ be the subgraph of $ D $ 
induced by $ X $. The 
\textbf{weak  components} of a digraph are the components 
of its underlying  undirected graph. We call a digraph weakly connected if it has just one weak component, i.e., its undirected 
underlying graph 
is connected. If $ x,y $ are vertices of the path $ P $, then we denote by $ \bm{P[x,y]} $ the segment of $ P $ between $ x $ and 
$ y $ (including $ x $ and $ y $).  For an undirected graph $ G $ and $ u\neq v\in V(G) $, we write $ 
\bm{\lambda_G(u,v)}$ for the \textbf{local edge-connectivity} between $ u $ and $ v $ in $ G $, i.e., the smallest 
cardinal $ \kappa 
$ such that  it is possible to delete $\kappa$ many edges in such a way that $ u $ and $ v $ are in different components of the resulting graph. Let us recall that the local 
edge-connectivity is also the maximal cardinal $ 
\kappa $ such that there is a system $ \mathcal{P} $ of pairwise edge-disjoint paths of size $ \kappa $ between $ u $ and $ v $.
We call a subset $ X $ of $ V $ \textbf{overloaded} (with respect to $ D $) if $ 
\left|\mathsf{out}_D(X)\right|<\left|\mathsf{in}_D(X)\right| $. A digraph $ D $ is called  \textbf{unbalanced} 
 if it admits an overloaded vertex set and $ D $ is \textbf{balanced} if it is not unbalanced.
 
The variables $\bm{\alpha} $ and $ \bm{\beta} $ are standing always for ordinal numbers while we use $ 
\bm{\kappa} $ and  $ \bm{\lambda} $ for cardinals. 
%The smallest limit ordinal, i.e., the set of the natural numbers is denoted by $ \bm{\omega} $.

\section{Elementary submodels and basic facts}

We give here a quick overview about elementary submodel techniques which play a central role in our proof. 
One can find a more detailed introduction with many combinatorial applications in \cite{soukup2011elementary}.

Roughly speaking, elementary submodels are sets which are closed under all possible (relevant) operations. Instead of ensuring 
the desired closures ``by hand''  it provides a flexible uniform framework giving effortlessly all the closures we  want to 
use. To make this precise, we need to apply some basic model theoretic concepts. All the formulas and models in this paper 
are 
in the 
first 
order language of set theory and the models are $ \in $-models, i.e., the ``element of'' relation in them is the real ``$ \in $''.  Let $ 
\Sigma=\{ \varphi_1,\dots, \varphi_n \} $ be a finite set of formulas where the free variables of $ \varphi_i $ are $ 
x_{i,1},\dots,x_{i,n_i} $. We call a set $ M $  a $ \bm{\Sigma} $\textbf{-elementary submodel} if $ 
\left|M\right|\subseteq M $ holds\footnote{The condition $ \left|M\right|\subseteq M $ is not 
always included in the definition of elementary submodels but it is a convenient assumption.} and the formulas in $ \Sigma $ are 
\textbf{absolute} between $ M $ and the universe, i.e., 
\[\bigwedge_{i=1}^{n} \left[ \forall x_1,\dots,  x_{n_i} \in M [(M 
\models \varphi_i(x_1,\dots,x_{n_i}))\Longleftrightarrow  \varphi_i(x_1,\dots,x_{n_i})] \right] . \] 
A fundamental fact we need (Corollary 2.6 in \cite{soukup2011elementary}) that one can find elementary submodels with certain 
prescribed parameters:
\begin{prop}\label{elemi reszm alap}
For every infinite cardinal $ \kappa $, every finite set $ \Sigma $ of formulas and every set $ x $  there exists a $ \Sigma 
$-elementary  submodel $ M $ of size $ \kappa $  with $ x\in M $.
\end{prop}

Defining a concrete $ \Sigma $ in a particular application is usually pointless. The common practise is ``pretending'' that $ \Sigma 
$ contains always what we need. These instances define implicitly afterwards what  the minimal sufficient  $ \Sigma $ is. Every $ 
\Sigma 
$-elementary submodel is closed under all operations $ F(x_1,\dots, x_n)=y $ defined by a formula  $ \varphi(x_1,\dots, 
x_n,y)  $ for which \[ \forall x_1\dots \forall x_n \exists! y \varphi(x_1,\dots, x_n,y),\  \varphi(x_1,\dots, x_n,y)\in \Sigma. \] 

Indeed, the first formula ensures that for every $ a_1,\dots, a_n\in M $ there is a $ b\in M $ such that  $ M\models 
\varphi(a_1,\dots, 
a_n,b)  
$  while the second guarantees that $ F(a_1,\dots, a_n)=b $.  To give a more explicit example, if a digraph $ D $ and one of 
its edges $ e $ are in the $ \Sigma $-elementary submodel $ M $, then the endpoints of $ e $ are also in $ M $ because they are 
definable from $ D $ and $ e $ and the corresponding defining formulas are assumed to be in $ \Sigma $. From now on whenever 
we claim that $ \Sigma $-elementary submodels have some specific properties, it is meant under the assumptions that $ \Sigma 
$ contains the necessary formulas.  

To illustrate arguments involving elementary submodels, we prove now three statements (which we need later anyway). Because 
of the introductory nature of this section  we give  more details in the proofs about 
how absoluteness is used and with which formulas than it is usual  in the normal practise. In order to keep the  arguments short, we will put some formulas ``redundantly'' in $\Sigma$ instead of using the fact that they are absolute even without being in $\Sigma$.

\begin{prop}\label{prop:smalsetSubset}
If $ M $ is a $ \Sigma $-elementary submodel and $ X\in M $ with $ \left|X\right|\leq\left|M\right| $, then $ X\subseteq M $.
\end{prop}
\begin{proof}
We assume that there are formulas in  $ \Sigma $ expressing the following:
\begin{enumerate}
\item\label{item: inCardinal op 1} For every $ X $ there is a cardinal $ \kappa $ such that $ \left|X\right|=\kappa $.
\item\label{item: inCardinal op 2}  
$ \left|X\right|=\kappa $.
\item\label{item: inCardinal op 3} 
$ f $ is a bijection  between $ X $ and $ Y$.
\item\label{item: inCardinal op 4}
$f\text{ is a function and } f(x)=y $.

\end{enumerate}
Since $ X\in M $ and formula \ref{item: inCardinal op 1} is in $ \Sigma $  there is a $ \kappa\in M $ such that $ M\models 
\left|X\right|=\kappa $. 
Because of formula \ref{item: inCardinal op 2} is in $ \Sigma $,  $
\left|X\right|=\kappa $ must hold. Then $ M\models ``\exists \text{ bijection from } \kappa\text{ to }  X $'' and therefore by 
basic logic there is some $ f\in M $ such that $ M\models ``f \text{ is a bijection from } \kappa \text{ to }  X  $''. As earlier, it 
must be true in the universe because formula \ref{item: inCardinal op 3} is in $ \Sigma $. It follows by basic logic that for every $\alpha\in M\cap \kappa$, there is a $y_\alpha \in M$ with $M\models f(\alpha)=y_\alpha$. Since formula \ref{item: inCardinal op 4} is in $\Sigma$,  $ f(\alpha)=y_\alpha $ holds for $\alpha \in M\cap \kappa$. But 
$ \kappa=\left|X\right|\subseteq \left|M\right|\subseteq M $ where the last inclusion was built in the definition of 
elementary submodels. Therefore $M\cap \kappa=\kappa$ and hence $f(\alpha)\in M$ for every $\alpha<\kappa$ from which  $X\subseteq M$ follows because $f: \kappa \rightarrow X$ is a bijection.

\end{proof}

If $ M $ is a $ \Sigma $-elementary submodel 
containing the directed or undirected graph $ G=(V,E) $, then we define $ \bm{G(M)}:=(V\cap 
M, E\cap M) $ and  $ \bm{G\bbslash M}:=(V,E\setminus M) $ which are subgraphs of $ G $.
\begin{prop}\label{nagyélöf}
Let $ G $ be an undirected graph an let $ M $ be a $ \Sigma $-elementary submodel with $ G\in M $.
Assume that $ \lambda_{G\bbslash M}(u,v)>0 $ for some $ u\neq v \in V(G)\cap M $. Then $ \lambda_{G}(u,v)>\left|M\right| $.
\end{prop}
\begin{proof}
We assume that $ \Sigma $ contains the formulas that expressing the following:
\begin{enumerate}
\item\label{item:elof 1}  $ (\forall G) (\forall u\neq v\in V(G)) \exists \kappa 
(\lambda_{G}(u,v)=\kappa) $.
\item\label{item:elof 1.5} $  \lambda_{G}(u,v)=\kappa $.
\item\label{item:elof 2} $ E'\subseteq E(G) $ separates the vertices $ u $ and $ v $ in graph $ G $.
\item\label{item:elof 3} $ \left|X\right|=\kappa $.

\end{enumerate}
Let $ u\neq v \in V(G)\cap M $ be arbitrary and suppose that $ \lambda_{G}(u,v)=:\kappa\leq\left|M\right| $. We have to show that 
$ \lambda_{G\bbslash M}(u,v)=0 $. Since $ G,u,v\in M $ and $ \kappa $ is definable from them and the formulas  
\ref{item:elof 1} and \ref{item:elof 1.5} are in $ \Sigma $, we know that $ \kappa\in M $ and $ M 
\models 
\lambda_{G}(u,v)=\kappa  $. Then there is some $ E' $ such that $ 
M\models ``E'\subseteq E(G) $ separates the vertices $ u $ and $ v $ in graph $ G $ and $ \left|E'\right|=\kappa $''. Formula 
\ref{item:elof 2}  ensures that $E'\subseteq E(G) $   separates the vertices $ u $ and $ v $ in graph $ G $ and 
formula \ref{item:elof 3} guarantees that $ \left|E'\right|=\kappa $.  But then by Proposition \ref{prop:smalsetSubset}, $ 
E'\subseteq M $ and  therefore $ 
\lambda_{G \bbslash M}(u,v)=0 $.
\end{proof}

We need the following result of L. Soukup (see \cite{soukup2011elementary}  Lemma 5.3 on p. 16):
\begin{prop}\label{egeszben szeparal}
Let $ G $ be an undirected graph and let $ M $ be a $ \Sigma $-elementary submodel with $ G\in M $. 
Assume 
that $ x\neq y\in V(G) $ are in the same component of $ G \bbslash M $ and $ F\subseteq E(G \bbslash M) $ separates them 
where $ 
\left|F\right| \leq \left|M\right|$. Then $ F $ separates $ x $ and $ y $ in the whole $ G $.
\end{prop}

\begin{proof}
Assume (reductio ad absurdum) that it is false and $ G,F,x,y,M $ witness it. We take a path $ P $ between $ x $ and $ y $ in $ 
G\bbslash 
F  $. Let $ x' $ and  $ y' $ be the first and the last intersection of $ P $ with $ V(G)\cap M $ with respect to some
direction of $ P $. 
The vertices $ x' $ and $ y' $ are well-defined and distinct since $ P $ necessarily uses some edge from $ E(G)\cap M $. We also fix a 
path 
$ Q $ between 
$ x $ and $ y $ in $ G\bbslash M $. The paths $ P[x',x],\ Q,\ P[y,y'] $  shows that $ x' $ and $ y' $ are in the same component of $ 
G\bbslash  M $. Thus by Proposition \ref{nagyélöf}, $ \lambda_G(x',y')>\left|M\right| $. There is a path $ R $ between $ x' $ 
and $ y' $ 
in $ G 
\bbslash  F $ 
since $ \lambda_G(x',y')> \left|M\right|\geq \left|F\right| $. But then $ P[x,x'],\ R,\ P[y',y] $ shows that $ F $ does not separate $ x 
$ and $ y $ in $ G \bbslash  M $ which 
is a contradiction.  
\end{proof}

\section{Proof of the main result}

\begin{proof}[Proof of Theorem \ref{main thm}]
In any digraph $ D=(V,E) $ for every  $ X\subseteq V $ the contribution of a directed cycle to $ 
\left|\mathsf{in}_D(X) \right|$ and $ \left|\mathsf{out}_D(X)\right| $ is the same,  thus if  $ 
E $ can be partitioned into directed cycles, then $ D $ must be balanced.

For countable digraphs the other direction of the equivalence is also easy. Let $ D=(V,E) $ be a balanced countable digraph.
Observe that  for each weak component $ X $,  $ D[X] $  must be strongly connected.  Thus every $ e\in E $ is in some 
directed cycle of $ D $. 
Note that, a balanced digraph remains balanced after the deletion of the edges of a directed cycle. We create 
a desired partition by recursion.
Let $ \prec $ be an $ \left|E\right| $-type ordering of  $ E $ and $ D_0:=D $.  In the $ n $-th step we take a directed cycle $ C_n $ 
in $ D_n $  through its $ \prec 
$-smallest edge and define $ D_{n+1}:=D_n\bbslash  E(C_n)$. Clearly, the resulting cycles $ C_n $ give a desired
partition.

For  uncountable digraphs the analogue of this naive recursive approach  does not work  because in a transfinite recursion one 
cannot ensure that after the first limit step the remaining digraph is still balanced.

\begin{lem}\label{main lemma}
For every infinite cardinal $ \kappa $ and  every set $ x $  there is a $ \Sigma $-elementary submodel $ M $ of size $ \kappa $ 
with $ x\in 
M $ such that for any balanced digraph $ D\in M $ the edge set $ E(D)\cap M$ can 
be partitioned into directed  cycles.
\end{lem}

Theorem \ref{main thm} follows directly from Lemma \ref{main lemma}: let $ D=(V,E) $ be an arbitrary balanced digraph and 
we use Lemma \ref{main lemma} 
with  $ x:=D $ and $ \kappa := \left|E\right|+\aleph_0 $. Then Proposition \ref{prop:smalsetSubset} guarantees $ E\subseteq 
M $ which ensures $ D=D(M) $. Hence  Lemma \ref{main lemma} gives a desired partition for $ D $ itself.

\begin{proof} 
We prove Lemma  \ref{main lemma} by  transfinite induction on $ \kappa $. Consider first the case  $ 
\kappa=\aleph_{0}$.  
Let  $ M $ be an arbitrary countable $ \Sigma $-elementary submodel  with $ x\in M $ (such an $ M $ exists by Proposition 
\ref{elemi reszm alap}). Assume that $ 
D=(V,E)\in M $ is a digraph such that $ E\cap M $ cannot be partitioned into directed cycles. We 
have to show that $ D $ is 
unbalanced. We know that $ D(M) $ must be unbalanced because it 
is countable and we have already proved Theorem \ref{main thm} for countable digraphs. Let $ X\subseteq V\cap M $ be an overloaded 
set in $ D(M) $. Then $ 
\mathsf{out}_{D(M)}(X)$ is  finite because  \[ \left|\mathsf{out}_{D(M)}(X)\right|< 
\left|\mathsf{in}_{D(M)}(X)\right|\leq 
\left|M\right|=\aleph_{0}. \] 
Let $ S $ be a set whose elements are the tails of the edges in $\mathsf{out}_{D(M)}(X) $ and the
heads of at least $\left| \mathsf{out}_{D(M)}(X)\right|+1 $ many edges from $ \mathsf{in}_{D(M)}(X) $.  Consider the set $ Y$ 
of vertices  that are 
reachable by a directed path  from $ S $ in  $ D $ without using any edges from $\mathsf{out}_{D(M)}(X) $. We show that $ Y 
$ is overloaded in $ D $.
It follows directly from the definition of $ Y $ that $\mathsf{out}_{D}(Y)\subseteq \mathsf{out}_{D(M)}(X) $. In order to show 
that $ \left|\mathsf{in}_{D}(Y)\right|\geq \left| \mathsf{out}_{D(M)}(X)\right|+1 $, let an $ e\in \mathsf{in}_{D(M)}(X) 
$ with 
its head in $ S $ be fixed. To guarantee that $ e\in \mathsf{in}_{D}(Y) $, we need to show that the tail  of $ e $ is not in $ Y $, 
i.e., it is not reachable from  $ S $ in $ D $ without using  edges from $\mathsf{out}_{D(M)}(X) $. Suppose for a 
contradiction that it is. Then by $ D,S, e, \mathsf{out}_{D(M)}(X)\in M $  there is a directed 
path $ P\in M$ witnessing this. But then by Proposition \ref{prop:smalsetSubset}, $ E(P)\subseteq M $ and hence $ P $ lies in $ 
D(M) $. Since $ P $ starts in $ X $ but terminates out of $ X $ it must use some edge from $ \mathsf{out}_{D(M)}(X) $ which is 
a contradiction. Since there are at least  $ \left| \mathsf{out}_{D(M)}(X)\right|+1 $ such an edge $ e $, $ 
\left|\mathsf{in}_{D}(Y)\right|\geq \left| \mathsf{out}_{D(M)}(X)\right|+1 $ follows. We can conclude that $ Y $ is 
an overloaded set in $ D $ and thus $ D $ is unbalanced.

Let $ \lambda>\aleph_{0} $ and assume that Lemma \ref{main lemma} is true for $ \kappa<\lambda 
$. We define a sequence of $ \Sigma 
$-elementary submodels $ \left\langle M_\alpha: 
\alpha<\lambda  \right\rangle $ by transfinite recursion such that for all $ \alpha<\lambda $: 

\begin{enumerate}
\item $ x\in M_\alpha $,
\item $ \left|M_\alpha\right| =\left|\alpha\right|+\aleph_0$,
\item $ \alpha, M_\alpha\in M_{\alpha+1} $,
\item\label{item: rec 4} if $ D\in M_{\alpha+1} $ is a balanced digraph, then the edge-set of $ D(M_{\alpha+1}) $ (i.e. $ E\cap 
M_{\alpha+1} $) 
can be 
partitioned into directed 
cycles,
\item $ M_\alpha=\bigcup_{\beta<\alpha}M_\beta $ if $ \alpha $ is a limit ordinal. 
\end{enumerate}

 Note that Proposition \ref{prop:smalsetSubset}  guarantees that $ M_\beta\subseteq M_\alpha $  for $ 
 \beta<\alpha<\lambda $. Let $ M_0 $ be an arbitrary  countable $ \Sigma $-elementary submodel containing $ x $. Suppose that 
 $ 
 M_\beta$ is already 
defined if $  \beta<\alpha $ for some $  \alpha<\lambda $ and 
satisfies the properties 
above. If $ \alpha $ is a limit ordinal, then our only choice is $ M_\alpha:=\bigcup\{ M_{\beta}: \beta<\alpha \} $. Then  $ 
M_\alpha $ is a $ \Sigma $-elementary 
submodel since it is the increasing union of  $ \Sigma $-elementary submodels\footnote{This implication is a basic fact from 
model theory, the 
proof is a straightforward formula induction.}. 
If  $ \alpha=\beta+1 $,  then we apply the induction hypothesis with cardinal $ \left|\alpha\right| 
+\aleph_0<\lambda$ and set $ \{ \beta, M_\beta \} $  to obtain an $ M_{\beta+1} $ satisfying the conditions.  The 
recursion is done.

 Let $ M:= \bigcup 
 \{ M_\alpha: 
\alpha<\lambda \} $. Then $ M $ is a $ \Sigma $-elementary submodel of size $\lambda $ and $ x\in  M $. 
 Let $ D\in M $ be a balanced digraph and let
$ \beta+1<\lambda $ be the smallest ordinal such that $ D\in M_{\beta+1} $. We define $ 
D_{\beta}$ to be $D(M_{\beta+1}) $ and for $ \alpha $ with $ 
\beta<\alpha<\lambda $ let $ D_\alpha:=(D\bbslash 
M_\alpha)(M_{\alpha+1}) $. These are edge-disjoint subgraphs of $ D(M) $, 
moreover, $ \{ E(D_\alpha): \beta\leq\alpha<\lambda \} $ is a partition of $ E(D)\cap M $. Since $ (D\bbslash M_\alpha) $ is 
definable from $ D, M_\alpha\in M_{\alpha+1}$, 
 we have $ (D\bbslash M_\alpha)\in M_{\alpha+1}  $.

\begin{claim}\label{D-M jo}
If  $ M $ is a $ \Sigma $-elementary submodel  and $ D\in M $ is a 
balanced digraph, then 
$D\bbslash M $ is also balanced.  
\end{claim}

If we  prove Claim  \ref{D-M jo}, then we are done with the proof  of Lemma \ref{main lemma} as well. Indeed, by Claim  
\ref{D-M jo}, the digraphs $ 
D\bbslash M_\alpha $ are balanced and therefore  by 
using property \ref{item: rec 4} with $ D\bbslash M_\alpha $ and  $ M_{\alpha+1} $ we can partition $ E(D_\alpha) $ into 
directed cycles for all $ \alpha $ with $  \beta\leq \alpha<\lambda $. By uniting these partitions,  we obtain a desired partition of $ 
E\cap M $.

Before we turn to the proof of Claim  \ref{D-M jo}, we need the following  observation to find overloaded sets in an 
unbalanced digraph with an extra property. 
\begin{prop}\label{elemi vag}
If $ D=(V,E) $ is an unbalanced digraph, then it has a weak  component  $ Z $ and a partition $ 
Z=X\cup Y $  such 
that $ D[X] $ and $ D[Y] $ are weakly connected 
and $ X $ is overloaded. 
\end{prop}
\begin{proof}
Let $ X'\subseteq V $ be  overloaded and let $ X_i\ (i\in I) $ be the  weak components of $ 
D[X'] $. Then \[ \sum_{i\in 
I}\left|\mathsf{out}_D(X_i)\right|=\left|\mathsf{out}_D(X')\right|<\left|\mathsf{in}_D(X')\right|=\sum_{i\in 
I}\left|\mathsf{in}_D(X_i)\right| \] and therefore there is an 
$ 
i_0\in I $ such that $ 
\left|\mathsf{out}_D(X_{i_0})\right|<\left|\mathsf{in}_D(X_{i_0})\right| $. Let $ Z $  be  the weak 
component of $ D $ that 
contains $ X_{i_0} $ and let $ Y_j\ (j\in J) $ be the weak 
components of 
$ D[Z\setminus 
X_{i_0}] $. 
Then
\[ \sum_{j\in 
J}\left|\mathsf{in}_D(Y_j)\right|=\left|\mathsf{out}_D(X_{i_0})\right|<\left|\mathsf{in}_D(X_{i_0})\right|=\sum_{i\in 
J}\left|\mathsf{out}_D(Y_j)\right| \]  and thus 
there is a 
$ j_0\in J $ such that $ 
\left|\mathsf{in}_D(Y_{j_0})\right|<\left|\mathsf{out}_D(Y_{j_0})\right| $.   Let $ Y:=Y_{j_0} $ and  $ X:= Z\setminus 
Y_{j_0} $. Then $ X $ is overloaded and both $ X=X_{i_0}\cup\{Y_j: j\in J\setminus\{ j_0 \}\} $ and $ Y=Y_{j_0} $ induce a 
weakly connected subdigraph in $ D $ because of the definition of the sets $ Y_j\ (j\in J) $. 
\end{proof}

\begin{proof}[Proof of Claim \ref{D-M jo}]\renewcommand{\qedsymbol}{}

Assume for contradiction, that $ D\bbslash  M $ is unbalanced. 
Then by Proposition \ref{elemi vag} there is a weak component $ Z $ of $ D\bbslash M $ with a partition
 $Z= X\cup Y $  such that $ (D\bbslash M)[X] $ and $ (D\bbslash M)[Y] $ are weakly connected and $ X $ is overloaded in $ 
 D\bbslash M $. 
Let $ F:=\mathsf{cut}_{D\bbslash M}(X) $. Next we show that $ \left|F\right| \leq \left|M\right|$. We may suppose that $ F $ 
is infinite and thus $ \mathsf{cut}_D(X) $ as well since $ F\subseteq \mathsf{cut}_D(X) $. Thus 
$ \aleph_{0}\leq \left|\mathsf{out}_D(X)\right|=\left|\mathsf{in}_D(X)\right| $.  
We must have $ \left|\mathsf{out}_{D\bbslash M}(X)\right|< \left|\mathsf{out}_D(X)\right|$ since otherwise  
\[ \left|\mathsf{out}_{D\bbslash M}(X)\right|= \left|\mathsf{out}_D(X)\right|=\left|\mathsf{in}_D(X)\right|\geq 
\left|\mathsf{in}_{D\bbslash M}(X)\right| \] which 
contradicts the 
choice of 
$ X $. Hence $ M $ 
contains  $ \left|\mathsf{out}_D(X)\right| $  elements of $ 
\mathsf{out}_D(X) $ and thus $ 
\left|\mathsf{out}_D(X)\right|\leq \left|M\right| $.  Then  
\[ \left|F\right|= 
\left|\mathsf{in}_{D\bbslash M}(X)\right|+\left|\mathsf{out}_{D\bbslash M}(X)\right|\leq 
\left|\mathsf{in}_{D}(X)\right|+\left|\mathsf{out}_{D}(X)\right|=\left|\mathsf{out}_{D}(X)\right|\leq 
\left|M\right|. \]

By using 
Proposition \ref{egeszben szeparal} to 
the undirected underlying graph of $ D $ with $ F $ and with arbitrary $ x\in X$ and $ y\in Y $,  we conclude that $ X $ and $ Y 
$ belong to distinct weak components  of $ D\bbslash F  $. Let us denote by $ X' 
$ and $ Y' $   these components respectively. We claim that $ \mathsf{cut}_{D}(X')=F $. Indeed, $ 
\mathsf{cut}_{D}(X')\subseteq F $ 
follows directly from the  definition of $ X' $, furthermore,  the elements of $ F $ go between $ X $ and $ Y $ and therefore 
between $ X' $ and $ Y' $. But then $ \left|\mathsf{out}_{D\bbslash M}(X)\right|=\left|\mathsf{out}_D(X')\right| $ and $ 
\left|\mathsf{in}_{D\bbslash M}(X)\right|=\left|\mathsf{in}_D(X')\right| $ thus \[ \left| 
\mathsf{out}_D(X')\right|=\left|\mathsf{out}_{D\bbslash M}(X)\right|<\left|\mathsf{in}_{D\bbslash  
M}(X)\right|=\left|\mathsf{in}_D(X')\right| \]

therefore $ X' $ is overloaded in $ D $  which is a contradiction.
\end{proof}
\end{proof}
\end{proof}

%%% AUTHOR: optional acknowledgments here
%\section*{Acknowledgments} %%  you may comment this out if no Ackno
%The authors are grateful to the anonymous reviewers for finding
%a bug in the main result.

%%% AUTHOR:
%%% Bibliography goes here. Note that the arXiv cannot process bibtex
%%% or biber bibliographies.  Example of acceptable bibliograpy format:
\bibliographystyle{amsplain}

\begin{thebibliography}{99}
\bibitem{nash1960decomposition}
{\sc Nash-Williams, C. S.~J.}
\newblock Decomposition of graphs into closed and endless chains.
\newblock {\em Proceedings of the London Mathematical Society 3}, 1 (1960),
  221--238.


\bibitem{thomassen2017nash}
{\sc Thomassen, C.}
\newblock Nash-williams’ cycle-decomposition theorem.
\newblock {\em Combinatorica 37}, 5 (2017), 1027--1037.

%\bibitem{nash1967infinite}
%{\sc Nash-Williams, C. S.~J.}
%\newblock Infinite graphs—a survey.
%\newblock {\em Journal of Combinatorial Theory 3}, 3 (1967), 286--301.

\bibitem{soukup2011elementary}
{\sc Soukup, L.}
\newblock Elementary submodels in infinite combinatorics.
\newblock {\em Discrete Mathematics 311}, 15 (2011), 1585--1598.


\end{thebibliography}

%% AUTHOR: You can generate such a bibliography from a .bib file by 
%% running pdflatex/bibtex/pdflatex/pdflatex and then pasting the .bbl file
%% between \begin{thebibliography} and \end{bibliography}

%%% AUTHOR: Include a short description of each author following the
%%% structure below. Use the same short tags used previously.  
%%% Use \imageat{} and \imagedot{} instead of "@" and "." in
%%% email addresses-this replaces the symbols with graphics to avoid 
%%% e-mail address harvesting from the .pdf file
\begin{aicauthors}
\begin{authorinfo}[jojo]
  Attila Jo\'o\\
  University of Hamburg and Egerv\'{a}ry Research Group
  on Combinatorial Optimization\\
  Hamburg, Germany\\
  attila\imagedot{}joo\imageat{}uni-hamburg\imagedot{}de \\
  \url{https://www.math.uni-hamburg.de/home/joo}
\end{authorinfo}
\end{aicauthors}

\end{document}